\documentclass{amsart}

\usepackage{graphics}
\usepackage{graphicx}
\usepackage{epsfig}
\usepackage{amscd}
\usepackage{amssymb}
\usepackage{amsthm}
\usepackage{amsmath}
\numberwithin{equation}{section}
\usepackage{mathrsfs}
\usepackage{hyperref}
\usepackage{eucal}
\usepackage{enumerate}  
\usepackage{bm}
\usepackage{float}
\usepackage{geometry}
\usepackage{tikz}
\usepackage{multirow}
\usepackage{verbatim}
\usepackage{geometry}
\usetikzlibrary{matrix}
\usetikzlibrary{arrows}
\usetikzlibrary{decorations.pathreplacing,decorations.markings}
\usetikzlibrary{matrix,arrows}
\usetikzlibrary{arrows,shapes,positioning}
\usetikzlibrary{decorations.markings}
\tikzstyle arrowstyle=[scale=1]
\tikzstyle directed=[postaction={decorate,decoration={markings,
    mark=at position .5 with {\arrow[arrowstyle]{stealth}}}}]
\tikzstyle reverse directed=[postaction={decorate,decoration={markings,
    mark=at position .65 with {\arrowreversed[arrowstyle]{stealth};}}}]

\newtheorem{thm}{Theorem}[section]

\newtheorem{cor}[thm]{Corollary}

\theoremstyle{plain}

\theoremstyle{remark}

\theoremstyle{definition}
\newtheorem{defn}[thm]{Definition}

\numberwithin{equation}{section}

\newcommand{\LL}{\mathcal{L}}
\newcommand{\A}{\mathcal{A}}
\newcommand{\B}{\mathcal{B}}
\newcommand{\R}{\mathbb{R}}
\newcommand{\C}{\mathbb{C}}

\begin{document}

\title{Topology of the icosidodecahedral arrangement}

\author{Ye Liu}
\address{Department of Mathematical Sciences, Xi'an Jiaotong-Liverpool University, Suzhou, China}
\email{yeliumath@gmail.com}



\date{}



\begin{abstract}
The icosidodecahedral arrangement is introduced by M. Yoshinaga \cite{Yoshinaga2019} as the first known example that is a hyperplane arrangement whose Milnor fiber has torsions in first integral homology. In this note, we prove that the icosidodecahedral arrangement is $K(\pi,1)$, hence so is its Milnor fiber.
\end{abstract}

\maketitle


\section{The icosidodecahedral arrangement}
A hyperplane arrangement $\A=\{H_1,\ldots,H_n\}$ in $V=\C^{\ell}$ is a finite collection of affine hyperplanes in $\C^{\ell}$. An arrangement $\A$ is said to be central if every hyperplane $H_i$ is a linear hyperplane, and $\A$ is said to be real if each hyperplane is defined by a linear form with real coefficients. Given a hyperplane arrangement $\A$, the intersection poset  $L(\A)=\{\cap_{H\in\B}H\neq\varnothing\mid \B\subseteq\A\}$ is the set of nonempty intersections of hyperplanes in $\A$, partially ordered by reverse inclusion. Note that $L(\A)$ has a unique minimal element $V$, which is the intersection over empty set. There is a rank function on $L(\A)$ defined by $r(X)=\mathrm{codim}(X)$. The rank $r(\A)$ of $\A$ is defined as the maximal value of $r$. For a central arrangement $\A$, $r(\A)=r(\cap_{H\in\A}H)$.

The complement $M(\A)=\C^{\ell}\setminus\cup_{H\in\A}H$ is an interesting topological space whose cohomology ring is known to be isomorphic to the Orlik-Solomon algebra \cite{Orlik-Solomon1980}, which is defined solely by $L(\A)$. In particular, the integral homology is torsion-free. Among other results, the Poincar\'e polynomial of $M(\A)$ can be computed from $L(\A)$
\[
\pi(\A,t)=\sum_{X\in L(\A)}\mu(X)(-t)^{r(X)},
\]
where $\mu$ is the M\"obius function on $L(\A)$.

Another interesting space associated to a central hyperplane arrangement is the Milnor fiber. See \cite{Dimca2017} for details. Suppose $\A$ is a central hyperplane arrangement in $\C^{\ell+1}$, then for each hyperplane $H\in\A$, we choose a linear form $\alpha_H$ with $H=\mathrm{Ker}~\alpha_H$. The homogeneous polynomial $Q=\prod_{H\in\A}\alpha_H$ is called the defining polynomial of $\A$. Note that $Q^{-1}(0)=\cup_{H\in\A}H$. The map $Q$ then restricts to a locally trivial fibration $M(\A)\to\C^*$ and the fiber $Q^{-1}(1)$ is called the \emph{Milnor fiber} of $\A$, denoted by $F(\A)$. Unlike $M(\A)$, the Milnor fiber $F(\A)$ is a rather less known space. For example, it is not known whether the Betti numbers of $F(\A)$ are determined solely by $L(\A)$. On the other hand, torsions have been found in the integral homology of $F(\A)$ for the case of \emph{multiple arrangements} \cite{Cohen-Denham-Suciu2003}, and for ordinary arrangements in higher homology \cite{Denham-Suciu2014}. 

M. Yoshinaga \cite{Yoshinaga2019} has recently found an interesting example whose Milnor fiber has torsions in first integral homology. The arrangement is (the complexification of) a real central (hyper)plane arrangement in $\R^3$, called the \emph{icosidodecahedral arrangement} for it fits perfectly in an icosidodecahedron. Denote this arrangement by $\A_{ID}$. We describe how to construct $\A_{ID}$ out of an icosidodecahedron, see Figure \ref{icosi}. The arrangement $\A_{ID}$ consists of 16 planes in $\R^3$ divided into two classes: edge planes and diagonal planes. An edge plane is a plane cutting the icosidodecahedron along an equator made of $10$ edges, there are $6$ such planes. A diagonal plane is a plane cutting the icosidodecahedron along a closed path made of $6$ diagonals of pentagon faces, there are $10$ such planes. A more useful way to view $\A_{ID}$ is to consider its deconing with respect to one of its plane. For example, we shall use the deconing $\LL_{ID}$ of $\A_{ID}$ with respect to an edge plane. To obtain the deconing, we put one edge plane as $\{z=0\}$ in $\R^3$, then take a section of $\A_{ID}$ by the plane $\{z=1\}$. In this way, we obtain the affine line arrangement $\LL_{ID}$ shown in Figure \ref{deconing}. The reverse construction is called \emph{coning}. Note that the blue lines in $\LL_{ID}$ come from edges planes of $\A_{ID}$ and the red lines come from diagonal planes. The computations about $\A_{ID}$, both in \cite{Yoshinaga2019} and this note, are carried out on $\LL_{ID}$. We shall prove that $\A_{ID}$ is a $K(\pi,1)$ arrangement, that is, the complement $M(\A_{ID})$ (of the complexification) is $K(\pi,1)$.

\begin{figure}[ht]
\centering
\includegraphics[height=100mm]{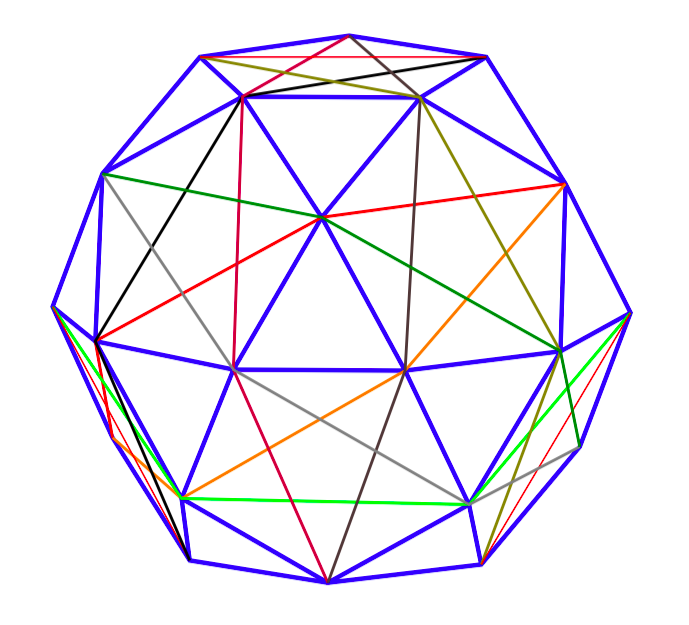}
\caption{The icosidodecahedral arrangement $\A_{ID}$}
\label{icosi}
\end{figure}

\begin{figure}[ht]
\centering
\includegraphics[height=100mm]{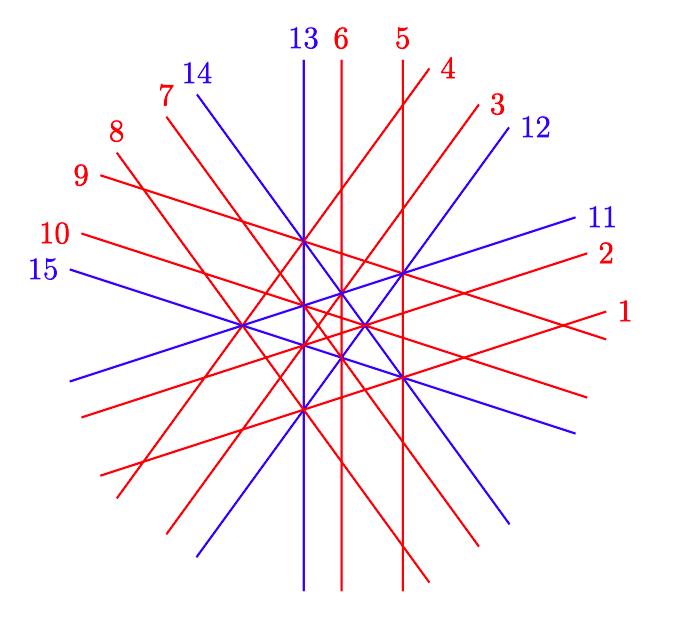}
\caption{The deconing $\LL_{ID}$ of $\A_{ID}$ with respect to an edge plane}
\label{deconing}
\end{figure}

\section{Properties that the icosidodecahedral arrangement does not enjoy}

For a real hyperplane arrangement, we may ask whether it is \emph{simplicial, supersolvable (fiber-type), free, $K(\pi,1)$, rational $K(\pi,1)$, factored,} satisfying the \emph{lower central series formula}, etc. In this section, we check that the icosidodecahedral arrangement fails all above properties except for the $K(\pi,1)$ property.

A real hyperplane arrangement is said to be \emph{simplicial} if each chamber can be transformed into the positive octant by a suitable coordinate change. P. Deligne \cite{Deligne1972} proved that simplicial arrangements are $K(\pi,1)$. But one easily sees that in the deconing $\LL_{ID}$ of $\A_{ID}$, there are pentagon chambers whose cone cannot be transformed into the octant $\R_{>0}^3$. Hence $\A_{ID}$ is not simplicial.

The Poincar\'e polynomial of $\LL_{ID}$ can be computed combinatorially
\[
\pi(\LL_{ID},t)=\sum_{X\in L(\LL_{ID})}\mu(X)(-t)^{r(X)}=1+15t+60t^2,
\]
and that of $\A_{ID}$
\[
\pi(\A_{ID},t)=(1+t)\pi(\LL_{ID},t)=1+16t+75t^2+60t^3.
\]
Note that $\pi(\A_{ID},t)$ does not factor as a product of degree one polynomials with integer coefficients. Thus $\A_{ID}$ is not free nor supersolvable (fiber-type) (see \cite{Terao1981, Jambu-Terao1984, Falk-Randell1985, Terao1986}). Furthermore, we have the following result.
\begin{thm}[\cite{Papadima-Yuzvinsky1999}, see also \cite{Falk-Randell2000}]
	For a rank $3$ arrangement $\A$, the following are equivalent
	\begin{itemize}
		\item $\A$ is fiber-type;
		\item the lower central series formula 
		\[
\pi(\A,-t)=\prod_{n\geq 1}(1-t^n)^{\phi_n}
\]
holds for $\A$, where $\phi_n$ is the rank of successive quotient $G_n/G_{n+1}$ of the lower central series $$\pi_1(M(\A))=G_1\supset G_2\supset \cdots$$
		\item $\A$ is rational $K(\pi,1)$.
	\end{itemize}
\end{thm}
We have concluded that $\A_{ID}$ is not rational $K(\pi,1)$, and the lower central series formula does not hold for $\A_{ID}$.

L. Paris \cite{Paris1995} proved that a \emph{factored} line arrangement and its cone are $K(\pi,1)$. However, $\LL_{ID}$ is not factored as we check below.
\begin{defn}
	A \emph{factorization} of an affine line arrangement $\LL$ is a partition $(\Pi_1,\Pi_2)$ of $\LL$ such that
	\begin{itemize}
		\item For any $\ell_1\in\Pi_1$ and $\ell_2\in\Pi_2$, $\ell_1\cap \ell_2\neq\varnothing$, and
		\item For any rank $2$ intersection $X\in L(\LL)_2$, either $\#(\Pi_1\cap\LL_X)=1$ or $\#(\Pi_2\cap\LL_X)=1$,
	\end{itemize}
	where $\LL_X=\{\ell\in\LL\mid \ell\supset X\}$. The arrangement $\LL$ is \emph{factored} is there is a factorization of $\LL$.
\end{defn}
We prove that $\LL_{ID}$ is not factored, that is, there is no factorization of $\LL_{ID}$. Suppose there is a factorization $(\Pi_1,\Pi_2)$ and $H_1\in\Pi_1$. Look at the multiplicity $2$ intersection $X=H_1\cap H_9$ and $(\LL_{ID})_X=\{H_1,H_9\}$. The second condition asserts that $H_9\in\Pi_2$. Similarly, by looking at the multiplicity $2$ intersection $H_3\cap H_9$, we see that $H_3\in\Pi_1$. But the intersection $H_1\cap H_3$ contradicts to the second condition.

\section{The icosidodecahedral arrangement is $K(\pi,1)$}
Although $\A_{ID}$ fails to satisfy many sufficient conditions for $K(\pi,1)$-ness, such as simplicial, supersolvable, factored conditions, we managed to prove $\A_{ID}$ is $K(\pi,1)$ using M. Falk's test \cite{Falk1995} for a central real 3-arrangement to be $K(\pi,1)$.

Let $\LL$ be an affine line arrangement in $\R^2$ and $\Gamma$ be the $2$-dimensional CW complex consisting of all bounded strata of $\R^2$ stratified by $\LL$. We use the terms vertices, edges and faces for $0$-, $1$- and $2$-dimensional cells, respectively. A \emph{corner} of $\Gamma$ is a pair $(v,f)$ where $v$ is a vertex of $\Gamma$ and $f$ is a face of $\Gamma$ such that $v$ lies in the closure of $f$. A \emph{system of weights} $\Delta$ on $\Gamma$ is a function on the set $\mathrm{Corner}(\Gamma)$ of corners of $\Gamma$ to $\R_{\geq 0}$. For a vertex $v$ of $\Gamma$, its \emph{link} $\Lambda_v$ is the simple graph whose vertices are edges of $\Gamma$ incident with $v$, and two vertices of $\Lambda_v$ are connected by an edge if the corresponding edges of $\Gamma$ are incident with a common face. Note that if $\Lambda_v$ is a cycle, then it contains $2m$ vertices, where $m$ is the multiplicity of $v$, i.e. the number of lines in $\LL$ through $v$.

A \emph{circuit} at $v$ is a closed walk on $\Lambda_v$, noting that a circuit may traverse an edge more than once. A system of weight $\Delta$ on $\Gamma$ induces a labeling of edges of $\Lambda_v$ for any vertex $v$. Furthermore, for a circuit $\xi$ at $v$, denote by $\Delta(\xi)$ the sum of labels of $\xi$. Note that in $\Delta(\xi)$, the labels are counted with multiplicities.
\begin{thm}[\cite{Falk1995}]\label{falk}
	If $\Gamma$ admits a system of weights $\Delta$ satisfying the following two conditions,
	\begin{enumerate}
		\item (Asphericity) For each face $f$ of $\Gamma$,
		\[
		\sum_{v:~(v,f)\in\mathrm{Corner}(\Gamma)}\Delta(v,f)\leq d(f)-2,
		\]
		where $d(f)$ is the number of vertices in the closure of $f$.
		\item ($\LL$-admissibility) For each vertex $v$ of $\Gamma$ with multiplicity $m$ and each of the following types of circuit $\xi$ at $v$, 
		\[
		\Delta(\xi)\geq 2.
		\]
		\begin{enumerate}[(i)]
		\item If $\Lambda_v$ is a cycle $(e_1,e_2,\ldots,e_k,e_1)$ ($e_i$ the vertices of $\Lambda_v$), then 
		$$\xi=(e_1,e_2,\ldots,e_k,e_1).$$
		
        \item If $\Lambda_v$ is a cycle $(e_1,e_2,\ldots,e_k,e_1)$ or a path $(e_1,e_2,\ldots,e_k)$ with $k>m+1$, then 
       $$
       \xi=(e_j,\ldots,e_{m+j},e_{m+j+1},e_{m+j},\ldots,e_j), ~1\leq j\leq k-m-1.
       $$
        \item If $\Lambda_v$ is a cycle $(e_1,e_2,\ldots,e_k,e_1)$ or a path $(e_1,e_2,\ldots,e_k)$ with $k>m$, then 
       $$\xi=(e_j,\ldots,e_{m+j},\ldots,e_j,\ldots,e_{m+j},\ldots,e_j),~1\leq j\leq k-m.$$
        \item If $\Lambda_v$ is a cycle $(e_1,e_2,\ldots,e_k,e_1)$ or a path $(e_1,e_2,\ldots,e_k)$ with $k>m$, then 
        $$\xi=(e_j,\ldots,e_{m+j},e_{m+j-1},e_{m+j},\ldots,e_j,e_{j+1},e_j),~1\leq j\leq k-m.$$
        \end{enumerate}
	\end{enumerate}
	then the cone $c\LL$ is a $K(\pi,1)$ arrangement.
\end{thm}

We are ready to use Falk's test to prove our result.
\begin{thm}
	The icosidodecahedral arrangement $\A_{ID}$ is $K(\pi,1)$, so is its Milnor fiber.
\end{thm}
\begin{proof}
	Let $\LL$ be the deconing $\LL_{ID}$ of $\A_{ID}$. We shall find an aspherical and $\LL$-admissible system of weights $\Delta$ on $\Gamma$. Using symmetry, the values of $\Delta$ are determined by those labeled in Figure 	\ref{sol}.
\begin{figure}[ht]
\centering
\includegraphics[height=100mm]{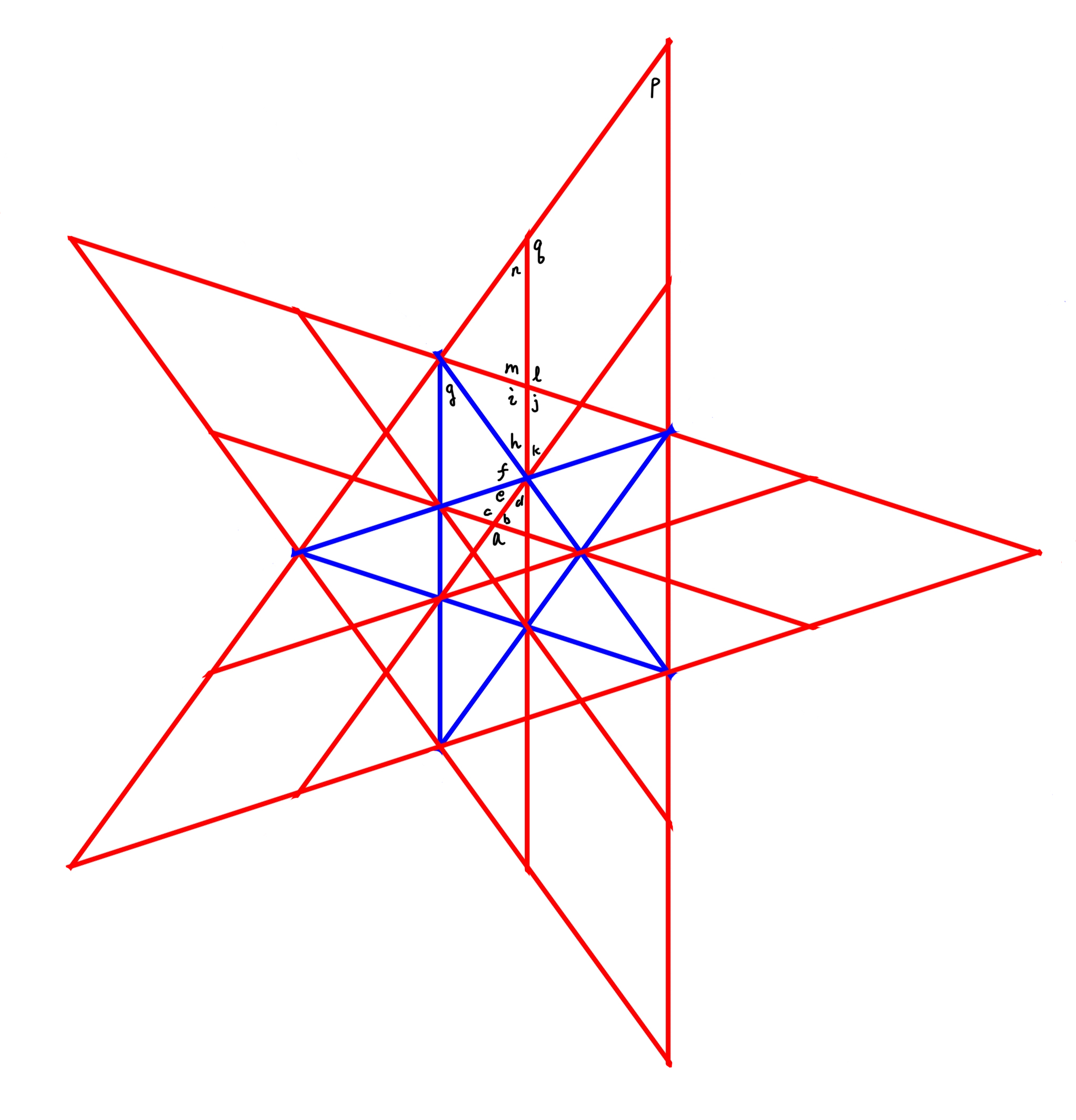}
\caption{A system of weights}
\label{sol}
\end{figure}
	
	We make a table of all shapes of $\Lambda_v$ and circuits appearing in this situation, where we use $123$ for abbreviation of $(e_1,e_2,e_3)$ and so on. Note that we have suppressed those circuits obtained by translations of indices.
	\begin{center}
    \begin{tabular}{ | l | l | l | l | l | l | l |}
    \hline
    $\Lambda_v$ & $m$ & Type (i) & Type (ii) & Type (iii) & Type (iv) \\ \hline
    12341 & $2$ & 12341  & 1234321 & 123212321 & 123232121 \\ \hline
    123456781 & $4$ & 123456781 & 12345654321 & 12345432123454321 & 1234545432121 \\ \hline
    123456 & $4$ & N/A & 12345654321 & 12345432123454321 & 1234545432121 \\ \hline
    123 & $2$ & N/A & N/A & 123212321 & 123232121 \\ \hline
    12 & $2$ & N/A & N/A & N/A & N/A\\
    \hline
    \end{tabular}
\end{center}
	
	The strategy is to attain equalities in the asphericity condition as much as possible. Here we list a valid solution $\Delta$
	\[
	a=c=\frac{3}{5}, ~b=f=h=j=m=p=\frac{2}{5},~d=e=g=i=k=n=\frac{1}{5},~\ell=1,~q=\frac{3}{10}.
	\]
	It is easily checked that all equalities in the asphericity condition hold in this case. To check the $\LL$-admissibility condition, thanks to symmetry, we only need to check $5$ vertices: those surrounded by weights (labels) $abcb, defhkhfe, mij\ell, mhghm, nq$. The links of the first three are cycles of length $4,8,4$ respectively, and the links of the last two are paths. Note that the vertex whose corner labeled by $p$ has link two vertices connected by an edge, which does not appear in the $\LL$-admissibility condition, so we can omit it. For the above five vertices, all or some of the four types of circuits will appear. Here we may check for one case as an example and the reader can easily complete the rest following the above table.
	
	Let us look at the vertex $v$ surrounded by $defhkhfe$. The link $\Lambda_v$ is a cycle of length $8$ with labels these letters. In this case, all four types of circuits appear. For type (i) circuits $\xi_1$, the sum $\Delta(\xi_1)=d+2e+2f+2h+k=\frac{12}{5}>2$. For type (ii) circuits $\xi_2$, the smallest possible sum $\Delta(\xi_2)=2(d+e+f+h+k)=\frac{14}{5}>2$. For type (iii) circuits $\xi_3$ and type (iv) circuits $\xi_4$, the smallest possible sums $\Delta(\xi_3)\geq \Delta(\xi_4)$ by their shapes, and $\Delta(\xi_4)=4e+2f+2h+4k=\frac{16}{5}>2$. Hence we have checked that the $\LL$-admissibility condition holds for the vertex $v$.
	
	Applying Falk's test (Theorem \ref{falk}), we conclude that $\A_{ID}$ is $K(\pi,1)$. The homotopy long exact sequence of the Milnor fibration proves that the Milnor fiber is also $K(\pi,1)$.
\end{proof}
There are some direct consequences, see \cite{Randell1997}.
\begin{cor}
The fundamental group $\pi_1(M(\A_{ID}))$ is of type FL and of cohomological dimension $3$. In particular, it is torsion-free. 
\end{cor}

\section*{Acknowledgement} I would like to thank Masahiko Yoshinaga for his interest toward this result and permission to use the figures from \cite{Yoshinaga2019}.

\renewcommand\refname{Reference}
\bibliographystyle{hep}
\bibliography{ico.bib}
\end{document}